\documentclass[11pt,reqno]{amsart}

\usepackage{amsmath,amssymb,amscd,graphicx}
\usepackage[all,2cell]{xy}
\usepackage{ifthen}
\usepackage{subsupscripts}
\usepackage[margin=1.25in]{geometry}	
\usepackage{url}
\usepackage{cite}
\usepackage{tikz}
\usepackage{hyperref}

\newboolean{sections}
\setboolean{sections}{true}

\UseComputerModernTips
\UseTwocells

\def\eoe{\unskip\ \hglue0mm\hfill$\diamond$\smallskip\goodbreak}

\makeatletter
\def\th@plain{%
  \thm@notefont{}
  \itshape 
}
\def\th@definition{%
  \thm@notefont{}
  \normalfont 
}
\makeatother

\newcommand{\calG}{\mathcal{G}}
\newcommand{\calH}{\mathcal{H}}

\newcommand{\CC}{{\mathbb{C}}}  
\newcommand{\NN}{{\mathbb{N}}}  
\newcommand{\QQ}{{\mathbb{Q}}}  
\newcommand{\RR}{{\mathbb{R}}}  
\newcommand{\ZZ}{{\mathbb{Z}}}  

\newcommand{\Ad}{{\operatorname{Ad}}}  
\newcommand{\Hom}{{\operatorname{Hom}}} 
\newcommand{\pr}{{\operatorname{pr}}} 
\DeclareMathOperator{\rank}{rank}  

\newcommand{\SO}{{\operatorname{SO}}}  
\DeclareMathOperator{\Spin}{Spin} 
\newcommand{\SU}{{\operatorname{SU}}}  
\newcommand{\U}{{\operatorname{U}}}  
\newcommand{\PO}{{\operatorname{PO}}}  
\newcommand{\rtlat}{Q}				
\newcommand{\cortlat}{\rtlat^\vee}			
\newcommand{\intlat}{\Lambda}				
\newcommand{\wtlat}{P}				
\newcommand{\cowtlat}{P^\vee}			

\newcommand{\toto}{{~\rightrightarrows~}} 
\newcommand{\inv}{{^{-1}}} 

\newcommand{\cxi}{\mathsf{i}} 

\makeatletter						
\newtheorem*{rep@theorem}{\rep@title}
\newcommand{\newreptheorem}[2]{
\newenvironment{rep#1}[1]{
\def\rep@title{#2 ##1}
\begin{rep@theorem}}
{\end{rep@theorem}}}
\makeatother

\newcommand{\ifsection}[2]{\ifthenelse{\boolean{sections}}{#1}{#2}}

\theoremstyle{plain}
\ifsection{
    \newtheorem{theorem}{Theorem}[section]
    
	\numberwithin{equation}{section}
	\numberwithin{figure}{section}
	\usepackage{chngcntr}
	\counterwithin{table}{section}
}
{
    \newtheorem{theorem}{Theorem}
}

\newreptheorem{theorem}{Theorem}		
\newreptheorem{proposition}{Proposition}
\newreptheorem{corollary}{Corollary}
\newtheorem{proposition}[theorem]{Proposition}

\newtheorem{corollary}[theorem]{Corollary}
\newtheorem{lemma}[theorem]{Lemma}

\theoremstyle{definition}
\newtheorem{definition}[theorem]{Definition}

\newtheorem{remark}[theorem]{Remark}

\newboolean{workmode}
\setboolean{workmode}{false}

 \newcommand{\pZ}{\operatorname{p}_Z}

\author{Derek Krepski}
\address{
Department of Mathematics, University of Manitoba, 
Winnipeg, MB, Canada 
}
\email{\href{mailto:Derek.Krepski@umanitoba.ca}{Derek.Krepski@umanitoba.ca}}
\urladdr{\href{http://server.math.umanitoba.ca/~dkrepski/}{http://server.math.umanitoba.ca/\textasciitilde dkrepski/}}

\title[Basic equivariant gerbes on simple Lie groups]{Basic equivariant gerbes on non-simply connected  compact simple Lie groups}
\date{\today}

\thanks{This work is partially supported by an NSERC Discovery Grant.}


\begin{document}

\begin{abstract}
This paper computes the obstruction to the existence of  equivariant extensions of basic gerbes over non-simply connected compact simple Lie groups. 
By modifying a (finite dimensional) construction of Gaw\c{e}dzki-Reis [\emph{J.\ Geom.\ Phys}.\ 50(1):28-55, 2004], we exhibit  basic equivariant bundle gerbes over non-simply connected compact simple Lie groups.
\end{abstract}

\maketitle
 

\section{Introduction} \label{s:intro}

Bundle gerbes were introduced by Murray \cite{murray1996bundle} as geometric models that represent cohomology classes in $H^3(M;\ZZ)$ of a smooth manifold $M$
(see also Chatterjee-Hitchin \cite{chatterjee1998construction,hitchin2001lectures}). Analogous to the classification of principal $S^1$-bundles over $M$ by their Chern class in $H^2(M;\ZZ)$, bundle gerbes over $M$ are classified by their \emph{Dixmier-Douady class} in $H^3(M;\ZZ)$. 
The study of gerbes in differential geometry began with  Brylinski \cite{brylinski2007loop} using the formalism of stacks, after Giraud \cite{giraud1971}. Other models for gerbes include central $S^1$-extensions of Lie groupoids \cite{behrend2011differentiable}, principal Lie 2-group bundles \cite{baez2007higher,nikolaus2013four,wockel2011principal}, and Dixmier-Douady bundles \cite{dixmier1963champs}.

Compact simple Lie groups provide a well known source of `naturally occurring' examples of  gerbes, and there exist various explicit constructions of these in the literature.    (The reader interested in a full chronology of these constructions is guided to the introduction of a paper by Murray-Stevenson \cite{murray2008basic}.)  
The present paper discusses  finite dimensional constructions of bundle gerbes of compact simple Lie groups that are equivariant with respect to the conjugation action.  To refine our discussion, we introduce some  additional details.

Let $G$ be a simply connected, compact, simple Lie group, acting on itself by conjugation.  A bundle gerbe  whose Dixmier-Douady class generates $H^3(G;\ZZ) \cong \ZZ$ is called a \emph{basic (bundle) gerbe} over $G$.  In \cite{meinrenken2003basic}, Meinrenken presents a finite dimensional construction of a basic  gerbe over $G$.  The resulting bundle gerbe is also strongly equivariant (see Definition \ref{d:eqbgerbe}), and since  $H^3_G(G;\ZZ) \cong H^3(G;\ZZ)$, it is a \emph{basic equivariant bundle gerbe} over $G$---that is, its Dixmier-Douady class generates $H^3_G(G;\ZZ)$.  (There is a more flexible notion of equivariance in the literature known as \emph{weak} equivariance (e.g.\ see \cite{murray2016equivariant,nikolaus2011equivariance}) where group elements act on the bundle gerbes by stable or Morita isomorphisms, up to coherent 2-isomorphisms. This added flexibility is not required in this work.)

For non-simply connected, compact, simple Lie groups $G'$, Gaw\c{e}dzki and Reis \cite{gawedzki2004basic} adapt  Meinrenken's construction to give a basic bundle gerbe over $G'$.  (Note that $H^3(G';\ZZ) \cong \ZZ$ for $G'\neq \PO(4n)$, whereas $H^3(\PO(4n);\ZZ) \cong \ZZ \oplus \ZZ_2$. By a basic gerbe for $\PO(4n)$, we shall mean a bundle gerbe whose Dixmier-Douady class corresponds to either $(1,0)$ or $(1,1)$ in $\ZZ \oplus \ZZ_2$ (\emph{cf}.\ Remark \ref{r:gerbeonPO})  However, the resulting bundle gerbe is not equivariant. In contrast to the simply connected case, since the natural map 
\begin{equation} \label{eq:naturalmap}
H^3_{G'}(G';\ZZ) \to H^3(G';\ZZ)
\end{equation}
 may fail to be surjective, a bundle gerbe over $G'$ need not admit an equivariant extension (i.e.\ an equivariant bundle gerbe with the same Dixmier-Douady class).  Additionally, the map \eqref{eq:naturalmap} fails to be injective (see \eqref{eq:exactseq}); therefore, if it all, a bundle gerbe over $G'$ admits multiple equivariant extensions.

The two main contributions of this paper are as follows.  First, we compute the obstruction to the existence of an equivariant extension of a basic bundle gerbe over $G'$ (i.e.\ we compute the image of the map \eqref{eq:naturalmap}) for each non-simply connected, compact simple Lie group. The obstruction amounts to a condition on the underlying \emph{level} of the gerbe over $G'$, as in the Theorem below. 

\begin{reptheorem}{\ref{t:obs}'}
Let $G'$ be a non-simply connected compact simple Lie group.  Let $\calG$ be a  bundle gerbe on $G'$ at level $\ell$. Then $\calG$ admits an equivariant extension if and only if $\ell$ is a multiple of the basic level $\ell_b$ of $G'$.
\end{reptheorem}

 Recall the level $\ell$ of a bundle gerbe $\calG$ over $G'$ is the integer $\ell=\mathrm{p}^*\mathrm{DD}(\calG)$ in $H^3(G;\ZZ) \cong \ZZ$, where $\mathrm{p}:G\to G'$ denotes the universal covering homomorphism. The basic level $\ell_b$ is the smallest multiple of the basic inner product on $\mathfrak{g}$ whose restriction to the integral lattice of $G'$ is integral (see Section \ref{ss:lie}). As an immediate Corollary, we obtain,
 
\begin{repcorollary}{\ref{c:eqautomatic}'}
The basic gerbe over  $G'$ admits an equivariant extension if and only if $G'$ is one of the following: $\SO(n)$, $\PO(8n+2)$, $\mathrm{PSp}(n)$, $\mathrm{Ss}(4n)$, $\mathrm{PE}_7$, or $\SU(N)/\ZZ_k$, where either $k^2|N$, or $k^2|2N$ and $N\equiv k\equiv2\mod 4$.
\end{repcorollary}

Second, we modify the construction of \cite{gawedzki2004basic} to give basic equivariant bundle gerbes over $G'$. This construction appears in Section \ref{ss:eqbasicGmodZ}.  Additionally, in Section \ref{ss:twist} we exhibit  a (finite dimensional) construction of the equivariant bundle gerbes whose Dixmier-Douady classes lie in the kernel of \eqref{eq:naturalmap}. Taking tensor products, this gives finite dimensional bundle gerbes representing all classes in $H^3_{G'}(G';\ZZ)$. 

Finally, we note that beyond the admittedly selective discussion above, there are other explicitly constructed equivariant gerbes  over non-simply connected groups appearing in the literature  (e.g.\ in \cite{brylinski2007loop,murray1996bundle,behrend2003equivariant}). Each of these constructions makes use of the path-fibration, and is thus infinite dimensional.  The finite dimensional construction in \cite{brylinski2000gerbes} gives an equivariant gerbe as a stack (i.e.\ sheaf of groupoids) over the Lie group.  Also, the constructions mentioned above all include the additional data of connective structures (i.e.\ a connection, curving, and 3-curvature).
 The present paper contributes an equivariant finite dimensional construction in the context of bundle gerbes; however, it does not include details regarding connective structures.

\bigskip

This paper is organized as follows.  In Section \ref{s:prelim}, we collect the notation and essentials of Lie theory used in describing the bundle gerbe constructions, as well as recall some background regarding the simplicial model for equivariant cohomology.  Section \ref{s:bgerbes} recalls some elementary definitions of bundle gerbes, in preparation for the constructions presented later in the paper.  In Section \ref{s:obs}, we compute the obstruction to finding an equivariant extension of gerbes on non-simply connected groups, and in Section \ref{s:eqbasicgerbe} we recall the construction of Gaw\c{e}dzki-Reis and present the necessary modifications to give an equivariant version.

\section{Preliminaries and notation} \label{s:prelim}

The (finite dimensional) construction of the basic bundle gerbe on a compact simple Lie group employs a reasonable amount of notation, mostly describing Lie-theoretic data.  This section collects the required notation used throughout the paper and recalls some relevant background.  

\subsection{Elementary Lie theory and lattices} \label{ss:lie}

Let $G$ be a compact, simply connected, simple Lie group with Lie algebra $\mathfrak{g}$.  We fix a maximal torus $T\subset G$ with Lie algebra $\mathfrak{t}$ and let $\mathfrak{t}^*$ denote its dual.  Let $\exp_G:\mathfrak{g} \to G$ with restriction $\exp=\exp_G|_\mathfrak{t}:\mathfrak{t} \to T$.

Let $\rtlat \subset \mathfrak{t}^*$ and 
$\cortlat\subset \mathfrak{t}$  denote the root and  coroot lattices, respectively, with dual lattices $\cowtlat \subset \mathfrak{t}$  and $\wtlat \subset \mathfrak{t}^*$, the coweight and weight lattices, respectively.

Recall that all compact simple Lie groups with Lie algebra $\mathfrak{g}$ are then of the form $G/Z$, where $Z$ is a subgroup of the centre $\mathcal{Z}(G)$, in which case we fix maximal tori $T/Z$.  Let $\pZ:G\to G/Z$ denote the quotient homomorphism.  Let $\intlat_Z = \ker \{ \pZ \circ \exp:\mathfrak{t} \to T/Z \}$ be the integral lattice of $T/Z$, which satisfies
\[
\cortlat \subset \intlat_Z \subset \cowtlat.
\]
Via the exponential map $\exp:\mathfrak{t} \to T$, the successive quotient lattices give the fundamental group $\pi_1(G/Z) \cong Z\subset T$ as well as the  centre of $G/Z$:
\[
Z \cong \intlat_Z/\cortlat, \quad \text{and} \quad \mathcal{Z}(G/Z) \cong \cowtlat/\intlat_Z.
\]
(In particular, $\cortlat = \ker \exp=\intlat_{\{1\}}$, and $\mathcal{Z}(G) \cong \cowtlat/\cortlat$.)  
Dually,  characters of the fundamental group are given by
\[
\Hom(Z,\U(1)) \cong (\cortlat)^*/(\intlat_Z)^*.
\]
Concretely, an element $\mu \in (\cortlat)^*$ gives the homomorphism $\mu(\exp(\zeta))=e^{2\pi \cxi \langle \mu,\zeta \rangle}$, for $\zeta \in \intlat_Z$.

 A choice of simple roots $\alpha_1, \ldots, \alpha_r$ (with $r=\rank(G)$) spanning $\rtlat$, determines the fundamental coweights $\lambda_1^\vee, \ldots, \lambda_r^\vee$ spanning $\cowtlat$,  defined by $\langle\alpha_i,\lambda_j^\vee \rangle=\delta_{i,j}$.  Similarly, the corresponding simple coroots $\alpha_1^\vee, \ldots, \alpha_r^\vee$  and fundamental weights $\lambda_1, \ldots, \lambda_r$ span the lattices $\cortlat$ and $\wtlat$, respectively. 
 
 Recall that non-zero elements of $\mathcal{Z}(G)$ can be written in the form $\exp(\lambda_i^\vee)$ where $\lambda_i^\vee$ is a fundamental coweight in $\cowtlat$ corresponding to a \emph{special root} $\alpha_i$ (i.e.\ a root $\alpha_i$ with unit coefficient in the expression $\tilde{\alpha}=\sum k_i\alpha_i$).

Let $\langle-,-\rangle$ denote  the \emph{basic inner product}, the invariant inner product on $\mathfrak{g}$ normalized to make short coroots have length $\sqrt{2}$.  With this inner product, we may sometimes identify $\mathfrak{t} \cong \mathfrak{t}^*$.  

An important choice underlying the construction of gerbes on compact simple Lie groups is the choice of \emph{level}, which is a parameter that ultimately scales the basic inner product.  In particular, we recall from \cite{laredo1999positive} the \emph{basic level} $\ell_b \in \NN$ of $G/Z$ is the smallest multiple $\ell$ such that $\ell\langle \zeta, \zeta' \rangle \in \ZZ$ for all $\zeta,\zeta' \in \intlat_Z$. \label{basiclevel}

\subsection{The fundamental alcove and centralizers} \label{ss:alcove}

The choice of simple roots together with the highest root $\tilde{\alpha}=:-\alpha_0$ determines a simplex in $\mathfrak{t}$ called the \emph{fundamental alcove},
\[
\Delta  = \{\xi \in \mathfrak{t} \colon  \langle \xi,\alpha_j\rangle \geq 0, \, \langle \xi, \tilde{\alpha}\rangle \leq 1 \}.
\]
Recall that $\Delta$ parametrizes conjugacy classes of $G$: each conjugacy class in $G$ contains a unique element of the form $\exp \xi$ with $\xi\in \Delta$. Let $q:G\to \Delta$ denote the quotient map (i.e.\ $q(g)=\xi$ if $g$ is conjugate to $\exp \xi$).

Denote the vertices of $\Delta$ by $0=\mu_0, \mu_1, \ldots, \mu_r$, where $\mu_j$ is the vertex opposite the facet parallel to $\ker \alpha_j$, 
and let $G_j\subset G$ denote the centralizer of $\exp \mu_j$.
For each $j=0,1,\ldots, r$, let $\Delta_j\subset \Delta$ be the complement of the closed facet opposite  the vertex $\mu_j$. More generally, for an index set $J\subset \{0,\ldots,r \}$, let $\Delta_J = \cap_{j\in J} \Delta_j$ and let $G_J$ denote the centralizer $G_{\exp \xi}$ for any (and hence all)  $\xi$ in the open face spanned by $\{\mu_j \, : \,  j\in J\}$.  

Let $\widetilde{G}_J \to G_J$ denote the universal covering homomorphism with kernel $\mathcal{Z}_J$. Note that for $I \subset J$, we have $G_J\subset G_I$, and hence a homomorphism $\widetilde{G}_J \to \widetilde{G}_I$ covering the inclusion. In particular, for any subset $I\subset J=\{0, \ldots, r\}$, let  $\widetilde{\exp}_I:\mathfrak{t} \to \widetilde{G}_I$ be the corresponding homomorphism.

Each vertex $\mu_j$ lies in $(\cortlat_j)^*$ \cite[Proposition 5.4]{meinrenken2003basic}, the lattice dual to the coroot lattice for $G_j$, and hence defines a character $\chi_j:\mathcal{Z}_j \to \U(1)$. Concretely,  for $\widetilde{\exp}_j(\xi) \in \mathcal{Z}_j$ we may write
 \(
\chi_j(\widetilde{\exp}_j(\xi)) = e^{2\pi \cxi \langle \mu_j , \xi \rangle}.
\)

For $i\neq j$, $\mu_j-\mu_i$ is fixed under conjugation by $\widetilde{G}_{ij}$ \cite[Lemma 5.5]{meinrenken2003basic}, and hence the formula
\[
\chi_{ij}(\widetilde{\exp}_{ij}(\xi))=e^{2\pi \cxi \langle \mu_j-\mu_i, \xi \rangle} \quad (\xi \in \mathfrak{t})
\]
defines a character on $\widetilde{G}_{ij}$.  Note that for $w\in \mathcal{Z}_{ij}$, which includes into both $\mathcal{Z}_i$ and $\mathcal{Z}_j$, we have 
\begin{equation} \label{eq:charcompat}
\chi_{ij}(w)=\chi_j(w)\chi_i(w)^\inv.
\end{equation}
Additionally, we also have that
\begin{equation} \label{eq:charcompat2}
\chi_{ij}\chi_{jk} = \chi_{ik}
\end{equation}
on $\widetilde{G}_{ijk}$ viewed as lying in $\widetilde{G}_{ij}$ and $\widetilde{G}_{jk}$.

\subsection{The action of the centre on the fundamental alcove} \label{ss:centre}

Let $W=N(T)/T$ denote the Weyl group.
Since translation by elements of $\mathcal{Z}(G)$ commutes with conjugation, there is an action $\mathcal{Z}(G)\times \Delta \to \Delta$, namely (see \cite[Propositions 4.1.2 and 4.1.4]{laredo1999positive}),
\[
z\cdot \xi = \Ad_{w_z^{-1}} \xi + \lambda_i^\vee,
\]
where $z=\exp(\lambda_i^\vee)$ and $w_z\in W$ is the unique element that leaves the set $\{\alpha_0, \ldots, \alpha_r\}$ invariant and satisfies $w_z(\alpha_0)=\alpha_i$. In particular, this defines a permutation of the vertices of $\Delta$, and equivalently the indices in $\{0, \ldots, r \}$
\[
z\cdot \mu_j = \mu_k \quad \Longleftrightarrow \quad w_z(\alpha_j) = \alpha_k \quad \Longleftrightarrow \quad z\cdot j =k.
\]
We will at times make use of the notation $i_z=z^{-1} \cdot i$ for indices $i$.

The above action is compatible with the notation of the centralizers and their covers; namely, for any subset $J\subset \{0, \ldots, r \}$,
\[
z\cdot \Delta_J=\Delta_{z\cdot J}, \quad \Ad_{w_z} G_J = G_{z\cdot J}, \quad \Ad_{w_z} \widetilde{G}_J = \widetilde{G}_{z\cdot J}.
\]

The map $z\mapsto w_z$ defines an injective homomorphism $\mathcal{Z}(G) \to W$, and we now choose representatives in $N(T)$, which by abuse of notation will also be denoted $w_z \in N(T)$.  The failure of the resulting lift $\mathcal{Z}(G) \to N(T)$ to be a homomorphism is measured by the $T$-valued 2-cocycle $c:\mathcal{Z}(G) \times \mathcal{Z}(G) \to T$,
\[
c_{z,z'} = w_z\, w_{z'} \, w_{zz'}\inv.
\]
We further choose elements $\xi_{z,z'} \in \mathfrak{t}$ such that $\exp(\xi_{z,z'}) = c_{z,z'}$ and for $J \subset \{0, \ldots, r \}$, let $\tilde{c}_{z,z'} \in \widetilde{G}_J$ be given by
\(
\tilde{c}_{z,z'} = \widetilde{\exp}_J(\xi_{z,z'}).
\)
(The dependence of $\tilde{c}$ on the subset $J$ will always be clear from the context.)

\subsection{Simplicial notation and equivariant cohomology} \label{ss:simp}

We recall some aspects related to the semi-simplicial manifold (degeneracy maps are not used in this paper) associated to a smooth $K$-action on a manifold $M$, where $K$ is a Lie group, which we then use to model the resulting $K$-equivariant cohomology of $M$.

Let $K$ be a Lie group, acting smoothly on a manifold $M$.  Recall the semi-simplicial manifold $(M_K)_\bullet$, with $(M_K)_n=  K^n \times M$ for $n\geq 0$, and face maps $(M_K)_{n} \to (M_K)_{n-1}$,
\[
\partial_i(g_1, \ldots, g_n,m) = \left\{ 
\begin{array}{ll}
(g_2, \ldots, g_n,m) & \text{if }i=0 \\
(g_1, \ldots, g_ig_{i+1},\ldots, g_n,m) & \text{if }0<i<n \\
(g_1, \ldots, g_{n-1},g_n\cdot x) & \text{if }i=n. \\
\end{array}
\right.
\]
(In other words, $(M_K)_\bullet$ is the simplicial manifold associated to the \emph{action groupoid} $K\ltimes M \toto M$ with source and target $\partial_0$ and $\partial_1$, respectively.)  
It is easily verified that the face maps satisfy the simplicial identities $\partial_i\partial_j = \partial_{j-1} \partial_i$ for $i<j$.

Let $(C^*,d)$ denote a presheaf of cochain complexes, and consider the double complex $C^*((M_K)_\bullet)$, depicted below.
\[
\xymatrix{
\vdots & \vdots & \vdots \\
C^2((M_K)_0)	\ar[r]^\partial \ar[u]^{d} & C^2((M_K)_1)\ar[r]^\partial \ar[u]^{-d} & C^2((M_K)_2) \ar[r]^\partial \ar[u]^{d}	& \cdots	\\
C^1((M_K)_0)	\ar[r]^\partial \ar[u]^{d} & C^1((M_K)_1)\ar[r]^\partial \ar[u]^{-d} & C^1((M_K)_2) \ar[r]^\partial \ar[u]^{d}	& \cdots	\\
C^0((M_K)_0)	\ar[r]^\partial \ar[u]^{d} & C^0((M_K)_1)\ar[r]^\partial \ar[u]^{-d} & C^0((M_K)_2) \ar[r]^\partial \ar[u]^{d}	& \cdots	\\
}
\]
The horizontal differential is the alternating sum , $\partial = \sum (-1)^i \partial_i^*$, of pullbacks along face maps.  Denote the total complex by 
\[
C(M_K):=\mathrm{Tot}(C^*((M_K)_\bullet)), \quad \text{with } C^n(M_K) = \bigoplus_{p+q=n} C^p((M_K)_q),
\]
with    total differential $\delta = (-1)^q d \oplus \partial$. 

For example, when $C^*=\Omega^*$ is the de~Rham complex, we obtain the Bott-Shulman-Stasheff complex of $(M_K)_\bullet$.  
In the present paper, we will apply the above formalism to the de~Rham complex $\Omega^*$, and smooth singular cochains $S^*(-;\ZZ)$ and $S^*(-;\RR)$.
Note that we will abuse notation and use integration of forms to view $\Omega^*(-)\subset S^*(-;\RR)$ and also view $S^*(-;\ZZ) \subset S^*(-;\RR)$ (e.g.\ we identify singular cohomology with real coefficients and de~Rham cohomology, $H^*(-;\RR) \cong H(\Omega(-))$).

Since the (fat) geometric realization of $(M_K)_\bullet$ is a model for the Borel construction $EK\times_K M$, the cohomology of the  complex $C(M_K)$ gives the corresponding $K$-equivariant cohomology of $M$.  In particular, we define $H^*_K(M;\RR) = H(\Omega(K_M),\delta)$ and $H^*_K(M;\ZZ)=H(S(M_K;\ZZ),\delta)$.

An \emph{equivariant extension} of a  cocycle $\omega \in C^*(M)$ is a cocycle $\omega_K \in C^*(M_K)$ that maps to $\omega$ under the natural  projection $C^*(M_K) \to C^*(M)$ (onto the first column, in the above schematic). Similarly, an equivariant extension of a cohomology class   is the cohomology class of an equivariant extension of any representative cocycle.  Note that if a cohomology class $[\omega]$ has an equivariant extension, then $\partial [\omega]=0$ in $H^3(C^*(M\times K),d)$.

We will be interested in (equivariant) cohomology of compact Lie groups in degree 3, of which we recall some facts for simple $G$.  Let $\theta^L, \theta^R \in \Omega^1(G;\mathfrak{g})$ denote the left and right invariant Maurer-Cartan forms (respectively), and denote the (bi-invariant) Cartan 3-form by  $\eta=\frac{1}{12} \langle \theta^L , [\theta^L,\theta^L] \rangle \in \Omega^3(G)$. Recall that the cohomology class $[\eta] \in H^3(G;\RR)\cong \RR$ is integral and generates $H^3(G;\ZZ)\cong \ZZ$.  Moreover, $\eta$ has an equivariant extension. \label{equivCartanform} Indeed, let $\omega \in \Omega^2(G\times G)$ be given at $(g,x)$ by
\[
\omega_{(g,x)} = -\frac{1}{2} \left( \langle \Ad_x \pr_1^*\theta^L,\pr_1^*\theta^L \rangle + \langle \pr_1^*\theta^L,\pr_2^*(\theta^L+\theta^R) \rangle
 \right).
\]
Then a direct calculation verifies that $\eta \oplus \omega$ is a cocycle in $\Omega^3(M_K)$ giving the required equivariant extension.

\section{Bundle gerbes and equivariant extensions} \label{s:bgerbes}

In this section, we collect some definitions and notation related to bundle gerbes, mainly for convenience, but also to set the framework for the explicit constructions in the following section.  For a detailed introduction to bundle gerbes, we refer the reader to Murray's original paper \cite{murray1996bundle} as well as \cite{murray2000bundle} and \cite{stevenson2000geometry}.

\begin{definition} \label{d:bgerbe}
Let $M$ be a manifold. A \emph{bundle gerbe on} $M$ is a triple $\calG=(X, L, \mu)$ consisting of a surjective submersion $\pi:X\to M$, a Hermitian line bundle $L\to X^{[2]}$, and a bundle isomorphism $\mu:\partial_0^*L \otimes \partial_2^*L \to \partial_1^*L$ over $X^{[3]}$ satisfying the associativity condition $\mu \circ (\mu\otimes 1) = \mu \circ (1 \otimes \mu)$ over $X^{[4]}$.
\end{definition}

For manifolds equipped with a smooth Lie group action, there are various notions of equivariance for bundle gerbes in the literature.  We will use the notion of \emph{strong} equivariance (e.g.\ see \cite{meinrenken2003basic}, \cite{stienon2010equivariant}) reviewed below, in which the group action on $M$ lifts to all the data of the bundle gerbe.  Though we shall not make use of it, we note that there is also a notion of \emph{weak} equivariance (e.g. see \cite{murray2016equivariant} and \cite{nikolaus2011equivariance}), where the group acts by stable or Morita isomorphisms up to coherent 2-isomorphisms.

\begin{definition} \label{d:eqbgerbe}
Let $K$ be a Lie group, acting smoothly on a manifold $M$.  A \emph{(strongly) $K$-equivariant bundle gerbe on $M$} is a bundle gerbe $\calG=(X,L,\mu)$ such that: $K$ acts on $X$ and $\pi:X\to M$ is $K$-equivariant; $K$ acts on $L$ and $L\to X^{[2]}$ is $K$-equivariant; and $\mu$ commutes with the $K$-action.
\end{definition}

Bundle gerbes $\calG$ on a manifold $M$ are classified by their \emph{Dixmier-Douady} class $\mathrm{DD}(\calG) \in H^3(M;\ZZ)$, analogous to the Chern-Weil classification of Hermitian line bundles by their Chern class in $H^2(M;\ZZ)$.  Similarly, $K$-equivariant bundle gerbes are classified by $H_K^3(M;\ZZ)$.  By a \emph{$K$-equivariant extension} of a bundle gerbe $\calG$ over $M$, we mean a $K$-equivariant bundle gerbe $\calG'$ with $\mathrm{DD}(\calG')=\mathrm{DD}(\calG)$. Equivalently, the $K$-equivariant Dixmier-Douady class of $\calG'$ is an equivariant extension of  $\mathrm{DD}(\calG)$.

Bundle gerbes can be pulled back along smooth maps of manifolds, and one can form tensor products of bundle gerbes.  Moreover, Dixmier-Douady classes are well-behaved with respect to pullbacks and tensor products.

\section{The obstruction to an equivariant extension for the basic gerbe} \label{s:obs}

This section  computes the cohomological obstruction to the existence of an equivariant extension of the basic gerbe on the compact simple Lie groups $G/Z$.  The vanishing of this obstruction class is a necessary condition for such an equivariant extension to exist, and by the construction in Section \ref{ss:basicGmodZ}, it is sufficient as well.  

Recall that by a \emph{basic} gerbe on $G/Z$, we mean a bundle gerbe $\calG$ on $G/Z$ whose DD-class 
\[
\mathrm{DD}(\calG) \in H^3(G/Z;\ZZ)\left\{\begin{array}{ll}\ZZ & \text{if }G/Z \neq \PO(4n) \\\ZZ\oplus \ZZ_2& \text{if }G/Z = \PO(4n)\end{array}\right.
\] 
generates the free part of $H^3(G/Z;\ZZ)$. The pullback $\pZ^*\mathrm{DD}(\calG)$ along $\pZ:G\to G/Z$ corresponds to $\ell_f \in \ZZ \cong H^3(G;\ZZ)$, where $\ell_f \in \{1,2\}$ is known as the \emph{fundamental level} of $G/Z$.  (The values of $\ell_f$ are known for each compact simple Lie group, e.g. see \cite{laredo1999positive,gawedzki2004basic}.) 
Generally, the \emph{level} $\ell$ of a bundle gerbe $\calH$ on $G/Z$ is given by $\ell=\pZ^*\mathrm{DD}(\calH)$ in $\ZZ\cong H^3(G;\ZZ)$.

If the bundle gerbe $\calG$ on $G/Z$ admits an equivariant extension, then  so does its DD-class---that is, $\mathrm{DD}(\calG)$ lies in the image of the natural map
\begin{equation} \label{eq:eqproj}
H^3_{G/Z}(G/Z;\ZZ) \to H^3(G/Z;\ZZ).
\end{equation}
Therefore, a necessary condition for the existence of such an extension is the vanishing of  the class
\(
\partial \, \mathrm{DD}(\calG) = \partial_0^*\mathrm{DD}(\calG) - \partial_1^*\mathrm{DD}(\calG) \in H^3(G/Z \times G/Z ;\ZZ),
\)
 where  $\partial_0=\pr_2$ and $\partial_1=\mathrm{Ad}$ are face maps of the semi-simplicial manifold $((G/Z)_{(G/Z)})_\bullet$ associated to the conjugation action.

Note that $\partial \mathrm{DD}(\calG)$ must be a torsion class (see Section \ref{s:prelim}, p.~\pageref{equivCartanform}).
It remains to compute the order of this torsion class for each compact simple Lie group.  As it turns out, this computation appears in \cite{krepski2008pre}, but in a different guise.  The following lemma recasts the obstruction class $\partial \mathrm{DD}(\calG)$ as the one appearing in \emph{loc.\ cit.}

\begin{lemma} \label{l:obs}
Let $K=G/Z$ be a compact simple Lie group, acting on itself by conjugation.  Let $\phi:K\times K \to K$ denote the commutator map, $\phi(g,h)=ghg^{-1}h^{-1}$. For cohomology classes $x\in H^3(K;\ZZ)$ corresponding to a generator of the $\ZZ$-summand,
\(
\partial_0^*x - \partial_1^*x = -\phi^*x.
\)
\end{lemma}
\begin{proof}
Note that $\phi$ may be expressed as the composition, 
\[
\xymatrix@C+20pt{
K\times K \ar[r]^{(\partial_1, c \,\circ \, \partial_0)} & K\times K \ar[r]^-{m} & K,
}
\]
where $c:K\to K$ denotes inversion, and $m:K\times K \to K$ denotes the group operation.

To verify the claim of the Lemma, we compute with coefficients in the fields $\QQ$ and $\ZZ_p$ for all primes $p$.  Over $\QQ$,  or over $\ZZ_p$ where $p$ does not divide the order of $\pi_1(K)\cong Z$, $x$ is primitive (i.e.\ $m^*x = x\otimes 1 + 1 \otimes x$).  Hence $c^*x=-x$, and
\begin{align*}
\phi^*x 	&= (\partial_1, c\, \circ\, \partial_0)^* m^* x \\
		&=\partial_1^*x - \partial_0^*x.
\end{align*}
For other primes $p$, $x$ may not  be primitive.  However, one can verify directly (i.e.\ case by case) that $m^*x = x\otimes 1 + 1 \otimes x + \sum a_i\otimes b_i$, where $a_i$ and $b_i$ are primitive. Therefore, $c^*b_i = -b_i$, and thus $c^*x = -x - \sum a_i c^*b_i = -x + \sum a_i b_i$.
Also, $\partial_1^*a_i=1\otimes a_i=\partial_0^* a_i$, and thus
\begin{align*}
\phi^*x 	
		&= (\partial_1, c\, \circ\, \partial_0)^* (x\otimes 1 + 1 \otimes x + \sum a_i\otimes b_i) \\
		&=\partial_1^*x - \partial_0^*x +  \partial_0^*(\sum a_ib_i) - \sum (\partial_0^*a_i )(\partial_0^*b_i ) \\
		&=\partial_1^*x - \partial_0^*x, 
\end{align*}
which verifies the claim.
\end{proof}

To compute the order of $\phi^*x$, first note that  $\phi:G/Z \times G/Z \to G/Z$ admits a canonical lift $\tilde\phi$ to the universal cover $G$,  
\[
\xymatrix{
							& G \ar[d]^\pZ \\
G/Z\times G/Z\ar[r]^-\phi \ar@{-->}[ur]^{\tilde{\phi}} & G/Z.
}
\]
By Lemma \ref{l:obs}, the order of the obstruction class $\partial \mathrm{DD}(\calG)$ equals the order of $\ell_f \cdot \tilde\phi^*z$, where $z$ generates $H^3(G;\ZZ)$.  The order of $\tilde\phi^*z$ was computed directly in \cite{krepski2008pre} (see also \cite[Proposition 4.1]{meinrenken2017verlinde}) for each compact simple Lie group, and is found to coincide with the \emph{basic level} $\ell_b$ (see Section \ref{s:prelim} p.\ \pageref{basiclevel}).
Hence, the obstruction $\phi^*x$ has order $\ell_b/\ell_f$.  This proves the `only if' part of the following Theorem.  (The other direction follows from the construction in Section \ref{s:eqbasicgerbe}.)

\begin{theorem} \label{t:obs}
Let $G/Z$ be a compact simple Lie group.  Let $\calG$ be a  bundle gerbe on $G/Z$ at level $\ell$. Then $\calG$ admits an equivariant extension if and only if $\ell$ is a multiple of the basic level $\ell_b$ of $G/Z$.
\end{theorem}

In particular, we see that the basic gerbe over $G/Z$ admits an equivariant extension if and only if $\ell_f=\ell_b$.  Using the Table in \cite[Proposition 3.6.2]{laredo1999positive}, we obtain the following corollary.

\begin{corollary} \label{c:eqautomatic}
The basic gerbe over $G/Z$ admits an equivariant extension if and only if $G/Z$ is one of the following: $\SO(n)$, $\PO(8n+2)$, $\mathrm{PSp}(n)$, $\mathrm{Ss}(4n)$, $\mathrm{PE}_7$, or $\SU(N)/\ZZ_k$, where either $k^2|N$, or $k^2|2N$ and $N\equiv k\equiv2\mod 4$.
\end{corollary}

\begin{remark} \label{r:basicmeaning}
In the literature (e.g.\ \cite{laredo1999positive}), the basic level $\ell_b$ usually refers to the smallest \emph{level} at which the loop group $L(G/Z)$ admits a central $\U(1)$-extension. 
That $\ell_b$ coincides with the order of the class $\tilde\phi^*z$, described in the paragraph preceding Theorem \ref{t:obs}, was noted in \cite{krepski2008pre} as a coincidence---that paper computes
computes the obstruction class $\tilde\phi^*z$ as the obstruction to the existence of a prequantization of the moduli space of flat $G/Z$-bundles on a closed Riemann surface.  A construction of a prequantization of this moduli space using loop groups appears in  \cite{krepski2010central}. 
\eoe
\end{remark}

\section{Equivariant bundle gerbes on compact simple Lie groups} \label{s:eqbasicgerbe}
 
This section reviews the construction of the basic gerbe on a compact simple Lie group, following the treatment of Gaw\c{e}dzki-Reis \cite{gawedzki2004basic} (see also \cite{meinrenken2003basic}), and provides the necessary modifications to make the construction equivariant with respect to the conjugation action.

\subsection{The basic gerbe} \label{ss:basicgerbe}

We briefly summarize the construction of the (non-equivariant) basic gerbe on a compact simple Lie group, following the treatment of  Gaw\c{e}dzki-Reis \cite{gawedzki2004basic}.  In the next subsection we will modify this construction to make it equivariant with respect to the conjugation action.

\subsubsection{Basic gerbe on $G$} \label{ss:basicG}
To begin, we recall the construction of  the basic gerbe $(X,P,\mu)$ on $G$, which is then used to construct the basic gerbe on quotients $G/Z$ in Section \ref{ss:basicGmodZ} below.
\medskip

For each $i\in \{0, \ldots, r \}$, let $V_i=q\inv(\Delta_i)$ denote the elements of $G$ that are conjugate to $\exp \xi$, for $\xi \in \Delta_i$, and  let $X_i=\{(g,h) \in V_i\times G \, | \, \rho_i(g)=hG_i \}$ be the pullback in the diagram, 
\[
\xymatrix{
X_i \ar[r] \ar[d] & V_i \ar[d]^{\rho_i} \\
G \ar[r] & G/G_i
}
\]
where $\rho_i:V_i \to G/G_i$ denotes the map $\rho_i(g\,\exp \xi \, g\inv)=gG_i$.  Hence $X_i \to V_i$ is a $G$-equivariant principal $G_i$-bundle---$G$ acts diagonally,
$$
a \cdot (g,h) = (aga^{-1},ah) \quad \text{for all }a\in G.
$$
More generally, for subsets $I\subset \{0, \ldots, r \}$, let $X_I \to V_I$ denote the analogous principal $G_I$-bundle.
Set $X=\sqcup_{i=0}^r X_i$, and let $\pi:X\to G$ be the submersion defined by setting $\pi|_{X_i}$ equal to $X_i\to V_i\hookrightarrow G$.  Hence $\pi:X\to G$ is a $G$-equivariant submersion.

To obtain the line bundle $P\to X^{[2]}$, we first describe the components of 
$$
X^{[2]}=\bigsqcup_{(i,j)}X^{[2]}_{ij}, \quad \text{i.e. } X^{[2]}_{ij}=X_i|_{V_{ij}}\times_{V_{ij}} X_j|_{V_{ij}},
$$ 
as orbit spaces, $ X^{[2]}_{ij}= \widehat{X}_{ij}/G_{ij}$, where 
\(
\widehat{X}_{ij}= X_{ij} \times G_i \times G_j,
\)
with $G_{ij}$-action by right translation on the last three factors. Note that the quotient map $\widehat{X}_{ij} \to X_{ij}$ given by $ (g,h,\gamma,\gamma')\mapsto (g,h\gamma^{-1}, g,h(\gamma')^{-1})$ is $G$-equivariant, where $G$ acts on $\widehat{X}_{ij}$ according to $a\cdot (g,h, \gamma, \gamma') = (aga^{-1}, ah, \gamma, \gamma')$.  

On each $\widehat{X}_{ij}$, we define a $G_{ij}$-equivariant line bunde $\widehat{P}_{ij} \to \widehat{X}_{ij}$,
\[
\widehat{P}_{ij} = (X_{ij} \times \widetilde{G}_i \times \widetilde{G}_j \times \CC ) / \sim
\]
where $(g,h,\tilde{\gamma},\tilde{\gamma}',u) \sim (g,h,\tilde{\gamma}w,\tilde{\gamma}'w',\chi_i(w)\chi_j(w')^{-1}u)$ for $w\in \mathcal{Z}_i$, $w' \in \mathcal{Z}_j$. The natural $G_{ij}$-action on $\widehat{P}_{ij}$ covering the one on $X_{ij}$ is given by 
\[
\beta\cdot [g,h,\tilde\gamma, \tilde\gamma', u] = [g,h\beta, \tilde\gamma \tilde\beta, \tilde\gamma' \tilde\beta, \chi_{ij}(\tilde\beta)\inv u ],
\]
where $\tilde\beta \in \widetilde{G}_{ij}$ is any lift (by \eqref{eq:charcompat}) of $\beta$.  
The restriction $P_{ij}$ of $P$ to $X^{[2]}_{ij}$ is then obtained by taking quotients,
\[
P_{ij}=\widehat{P}_{ij}/G_{ij} \to \widehat{X}_{ij}/G_{ij}=X^{[2]}_{ij}.
\]
Notice that the $G$-action on $X_{ij}$ readily lefts to $P_{ij}$,
\[
a \cdot [g,h,\tilde{\gamma},\tilde{\gamma}',u] = [aga^{-1},ah,\tilde{\gamma},\tilde{\gamma}',u].
\]

We describe the groupoid multiplication $\mu:\partial_0^*P \otimes \partial_2^*P \to \partial_1^* P$ fibrewise over  
\[
X^{[3]} = \bigsqcup_{(i,j,k)} X^{[3]}_{ijk}, \quad \text{with } X^{[3]}_{ijk} =  X_i|_{V_{ijk}}\times_{V_{ijk}}X_j|_{V_{ijk}}\times_{V_{ijk}}X_k|_{V_{ijk}}.
\]
To that end, note that (similar to the decomposition for $X^{[2]}$ above) the components $X^{[3]}_{ijk}$ are orbit spaces $X^{[3]}_{ijk}=\widehat{X}_{ijk}/G_{ijk}$, where $\widehat{X}_{ijk}=X_{ijk}\times G_i \times G_j \times G_k$, with $G_{ijk}$-action by right translation on the last four factors.  
Let  $(x_1,x_2,x_3) \in X^{[3]}_{ijk}$ and write $(x_1, x_2, x_3) = [g,h,\gamma,\gamma',\gamma']$. For elements
\[
s=[g,h,\tilde\gamma,\tilde\gamma',u]  \in P_{(x_1,x_2)}\quad \text{and} \quad t=[g,h,\tilde\gamma',\tilde\gamma'',v] \in P_{(x_2,x_3)},
\]
we set 
\[
\mu(s\otimes t) = [g,h,\tilde\gamma,\tilde\gamma'',uv] \in P_{(x_1,x_3)}.
\]
It is straightforward to verify (using \eqref{eq:charcompat2}) that this is well-defined and satisfies the associativity condition over $X^{[4]}$.

\subsubsection{Basic gerbe on $G/Z$} \label{ss:basicGmodZ}
We now turn to the basic gerbe  $(Y,Q,\nu)$ on $G/Z$.  Let $\pi:X\to G$ be the submersion in Section \ref{ss:basicG}, and observe that translation in $G$ by $z\in Z$ lifts to a map in $X$,
\begin{equation} \label{eq:transl}
z\cdot (g,h) = (zg, hw_z\inv),
\end{equation}
 which shuffles the components, $z\cdot X_i =X_{z\cdot i}$. 
(As noted in \cite{gawedzki2004basic}, the lifts do not extend to a simplicial map on $X$ since $z\mapsto w_z \in N(T)$ is not necessarily a homomorphism.)

Let $Y=X$, and consider the submersion $\pZ \circ \pi:Y\to G/Z$.  We shall abuse notation and denote this submersion by $\pi$ as well.

To obtain the line bundle $Q\to Y^{[2]}$, we first give a decomposition of $Y^{[2]}$.  
For $(g_1, h_1; g_2, h_2) \in Y^{[2]}$, there exists $z$ satisfying $g_1=zg_2$, which gives the identification
\begin{equation} \label{eq:decompY2}
Y^{[2]} = X^{[2]} \times Z, \quad (g_1, h_1; g_2, h_2) \mapsto (g_1,h_1; zg_2, h_2 w_z^{-1};z).
\end{equation}
Using the above identification, write $Q=P^{\ell_f}\! \times Z$, where $P \to X^{[2]}$ is the line bundle from the basic gerbe $(X,P,\mu)$ on $G$. (In particular, one replaces the characters $\chi_i$ and $\chi_{ij}$ in the  construction of \ref{ss:basicG} with their $\ell_f$-th powers.)  Write the restriction of $Q$ to a component of $Y^{[2]}$ as 
\[
Q_{ij;z}=Q|_{X_{ij}^{[2]} \times \{z\}} \cong P_{ij}^{\ell_f}.
\]

As in Section \ref{ss:basicG}, we describe the groupoid multiplication $\nu:\partial_0^*Q \otimes \partial_2^*Q \to \partial_1^*Q$ over $Y^{[3]}$ fibrewise.  Let  $(y_1, y_2, y_3)=(g_1,h_1;g_2,h_2;g_3,h_3) \in  Y^{[3]}$, and let $z,z'\in Z$ so that
\[
(g_1, h_1; zg_2, h_2 w_z\inv; zz'g_3, h_3 w_{z'}\inv w_z\inv) \in X^{[3]}_{ijk}.
\]
Therefore, under the identification \eqref{eq:decompY2}, $(y_1,y_2)$ lies in (the component corresponding to) $X^{[2]}_{ij} \times \{z\}$, $(y_2,y_3)$ lies in   $X^{[2]}_{j_z k_z} \times \{z'\}$, and $(y_1,y_3)$ lies in $X^{[2]}_{ik} \times \{zz'\}$. 
Hence, on fibres, the groupoid multiplication $\nu$ takes the form,
\[
Q_{ij;z}\big|_{(y_1,y_2)}\otimes Q_{j_z k_z;z'}\big|_{(y_2,y_3)} \longrightarrow Q_{ik;zz'}\big|_{(y_1,y_3)}.
\]

Let $g=g_1$, and choose $(h,\gamma,\gamma', \gamma'')$ in $G\times G_i \times G_j \times G_k$ (unique up to translation by an element of $G_{ijk}$) so that  
\[
(g_1, h_1; zg_2, h_2 w_z\inv; zz'g_3, h_3 w_{z'}\inv w_z\inv) =(g,h\gamma\inv; g, h\gamma'\inv; g, h\gamma''\inv).
\]
Writing $\gamma_z = \Ad_{w_z\inv} \gamma$, we see that under \eqref{eq:decompY2},
\begin{align*}
(y_1,y_2) & \longleftrightarrow (g,h\gamma\inv; g,h\gamma'\inv;z), \\
(y_2,y_3) & \longleftrightarrow (z\inv g,h w_z \gamma_z'\inv; z\inv g,h w_z \gamma_z''\inv;z'), \text{ and}\\
(y_1,y_3) & \longleftrightarrow (g,h\gamma\inv; g,h\gamma''\inv c_{z,z'};zz'). 
\end{align*}
Finally, for elements
\begin{align*}
s&= ([g,h,\tilde\gamma,\tilde\gamma',u],  z)\in Q_{ij;z}\big|_{(y_1,y_2)},  \\
t&= ([z\inv g, hw_z, \tilde\gamma_z',\tilde\gamma_z'',v] , z')\in Q_{j_z k_z;z'}\big|_{(y_2,y_3)},
\end{align*}
we set 
\begin{equation} \label{eq:multGR}
\nu(s\otimes t) =( [g,h,\tilde\gamma,\tilde{c}_{z,z'}\inv\tilde\gamma'', u_{z,z'}^{ijk}\, uv] , zz')\in Q_{ik;zz'}\big|_{(y_1,y_3)},
\end{equation}
where the constants $u_{z,z'}^{ijk} \in \U(1)$ must satisfy a certain constraint so that $\nu$ satisfies the associativity condition.  As shown in \cite{gawedzki2004basic}, there exists such constants (satisfying said constraint) whenever one chooses $Q=P^\ell \times Z$ with $\ell$ a multiple of the fundamental level $\ell_f$. (More precisely, for each compact simple Lie group $G/Z$, Gaw\c{e}dzki and Reis \cite{gawedzki2004basic} calculate $\ell_f$ as the smallest $\ell$ for which the constraints on the constants $u_{z,z'}^{ijk}$ are satisfied.)

\begin{remark} \label{r:gerbeonPO}
Recall that $H^3(\PO(4n);\ZZ) \cong \ZZ \oplus \ZZ_2$, and hence there two (non-isomorphic) basic bundle gerbes.  In \cite{gawedzki2004basic}, the authors show that there are two non-equivalent choices of constants $u_{z,z'}^{ijk}$ defining the gerbe multiplication, ultimately related to the two isomorphism classes of central $\U(1)$-extensions of $Z\cong \ZZ_2\times \ZZ_2$.  

An alternative way to realize the torsion element in $H^3(\PO(4n);\ZZ)$ is as follows. Let $\alpha:Z\times Z \to \U(1)$ be a group 2-cocycle defining a group operation on $\widehat{Z}=Z\times \U(1)$ by $(z_1,\lambda_1)(z_2,\lambda_2)=(z_1z_2,\alpha(z_1,z_2)\lambda_1\lambda_2)$.  
Consider the bundle gerbe $(\Spin(4n),R,\sigma)$, where $\Spin(4n) \to \PO(4n)$ is the universal covering homomorphism, with trivial line bundle $R=(\Spin(4n)\times_{\PO(4n)} \Spin(4n)) \times \CC$.  We define the gerbe multiplication fibrewise, 
\(
\sigma: R_{(g_1,g_2)} \otimes R_{(g_2,g_3)} \to R_{(g_1,g_3)},
\)
by 
\[
\sigma ( (g_1,g_2,\lambda) \otimes (g_2,g_3,\lambda')) = (g_1,g_3,\alpha(g_2g_1^{-1},g_3g_2^{-1})\lambda\lambda'). 
\]
Associativity of the gerbe multiplication follows from the fact that $\alpha$ is a 2-cocycle (i.e.\ defines an associative group operation on $\widehat{Z}$).
\eoe
\end{remark}

\subsection{The basic equivariant gerbe on $G/Z$} \label{ss:eqbasicGmodZ}

In this Section, we modify the construction of \cite{gawedzki2004basic}, summarized in \ref{ss:basicGmodZ}, to exhibit a basic $G/Z$-equivariant  gerbe $(Y',Q',\nu')$ over $G/Z$.

For each $i\in \{0, \ldots, r \}$ observe that   $Z\subset \mathcal{Z}(G)\subset T\subset G_i$ and hence the natural quotient map $G\to G/G_i$ factors through $G/Z$. Let $Y_i'$ denote the pullback
\[
\xymatrix{
Y_i' \ar[r]^{\pi_i'} \ar[d] & V_i \ar[d]^{\rho_I} \\
G/Z \ar[r] & G/G_i
}
\]
so that $Y_i'=\{(g,hZ) \in V_i\times G/Z \, | \, \rho_i(g)=hG_i \}$.   Notice that the $G$-action on $X_i=Y_i$ descends to $Y_i'$:
\[
a\cdot (g,hZ)=(aga^{-1},ahZ),
\]
and its restriction to $Z$ is trivial. Therefore, $G/Z$ acts on $Y_i'$ and $\pi_i'$ is $G/Z$-equivariant.
Letting $\pi':Y'=\sqcup Y_i' \to G/Z$ with $\pi'|_{Y_i'}=\pi'_i$, we obtain a $G/Z$-equivariant submersion.

Equivalently, consider the $Z$-action on $Y$ given by $z(g,h)=(g,zh)$.  Since $\pi:Y\to G/Z$ is $Z$-invariant, it descends to the quotient $\pi':Y' \to G/Z$. Additionally, it is straightforward to see that $Y^{[2]}$ inherits a natural $Z\times Z$-action, and that $(Y')^{[2]} = Y^{[2]}/(Z\times Z)$.

 To obtain the line bundle  $Q'\to (Y')^{[2]}$, we first lift the $Z\times Z$-action on $X^{[2]}_{ij}$ to the $\ell$-th power $P^\ell_{ij}$.
 \begin{proposition} \label{p:Z2actionlift}
 For $(z_1,z_2) \in Z\times Z$, choose $\zeta_1,\zeta_2 \in \intlat_Z\subset \mathfrak{t}$ with $\exp(  \zeta_j)=z_j$. The formula
\[
(z_1,z_2) \cdot [g,h,\tilde{\gamma},\tilde{\gamma}',u] =  [g,h,\tilde{\gamma}\, \widetilde{\exp}_i{ \zeta_1},\tilde{\gamma}'\, \widetilde{\exp}_j{\zeta_2},e^{2\pi \cxi \, \ell(\langle \mu_i,\zeta_1\rangle-\langle\mu_j,\zeta_2\rangle)}u]
\]
defines a $Z\times Z$-action on $P^\ell_{ij}$ lifting the $Z\times Z$-action on $X^{[2]}_{ij}$.
 \end{proposition}
 \begin{proof}
The above formula does not depend on the choice of $\zeta_j$ representing $z_j$, since for $\xi_1,\xi_2 \in \intlat$, we have that $\, \widetilde{\exp}_i{\xi_i} \in \mathcal{Z}_i$ and $\, \widetilde{\exp}_j{\xi_i} \in \mathcal{Z}_j$ so that $\chi_i^\ell(\, \widetilde{\exp}_i{\xi_1})= e^{2\pi \cxi \, \ell\langle \mu_i, \xi_1\rangle}$ and $\chi_j^\ell(\, \widetilde{\exp}_j{\xi_2})= e^{2\pi \cxi \, \ell\langle \mu_j, \xi_2\rangle}$.  Therefore,
\begin{align*}
[g,h&,\tilde{\gamma}\, \widetilde{\exp}_i{(\zeta_1+\xi_1)},\tilde{\gamma}'\, \widetilde{\exp}_j{(\zeta_2+\xi_2)},e^{2\pi \cxi \, \ell(\langle \mu_i,\zeta_1+\xi_1\rangle-\langle\mu_j,\zeta_2+\xi_1\rangle)}u] \\
&=[g,h,\tilde{\gamma}\, \widetilde{\exp}_i{\zeta_1}\, \widetilde{\exp}_i{\xi_1},\tilde{\gamma}'\, \widetilde{\exp}_j{\zeta_2}\, \widetilde{\exp}_j{\xi_2}, \\
& \qquad e^{2\pi \cxi \, \ell(\langle \mu_i,\zeta_1\rangle-\langle\mu_j,\zeta_2\rangle)} e^{2\pi \cxi \, \ell\langle \mu_i, \xi_1\rangle}e^{-2\pi \cxi \, \ell \langle \mu_j, \xi_2\rangle}u] \\
&= [g,h,\tilde{\gamma}\, \widetilde{\exp}_i{\zeta_1},\tilde{\gamma}'\, \widetilde{\exp}_j{\zeta_2},e^{2\pi \cxi \, \ell(\langle \mu_i,\zeta_1\rangle-\langle\mu_j,\zeta_2\rangle)}u],
\end{align*}
as required.
\end{proof}

Altogether, the $P^\ell_{ij}$ assemble to give the  $Z\times Z$-equivariant line bundle $P^\ell  \to X^{[2]}$. Since the $Z\times Z$-action on $X^{[2]}$ is free, it is also free on $P^\ell$, and we obtain a line bundle $P^\ell/(Z\times Z) \to X^{[2]}/(Z\times Z)$.  Similarly, for $P^\ell \times Z \to X^{[2]} \times Z = Y^{[2]}$, and we thus obtain a line bundle 
$$
Q'=(P^\ell/(Z\times Z)) \times Z \to (Y')^{[2]}.
$$

Analogous to \eqref{eq:decompY2}, we have the decomposition
\begin{equation} \label{eq:decompY2'}
(Y')^{[2]} =  (X^{[2]}/Z^2) \times Z, \quad (g_1,h_1Z;g_2,h_2Z) \mapsto (g_1,h_1Z; zg_2, h_2 w_z^{-1}Z;z),
\end{equation}
which we will use to define the groupoid multiplication $\nu':\partial_0^*Q' \otimes \partial_2^* Q' \to \partial_1^* Q'$ fibrewise.
Let $(y_1, y_2, y_3) = (g_1, h_1 Z; g_2, h_2 Z; g_3, h_3 Z) \in (Y')^{[3]}$, and let $z,z' \in Z$ so that
\[
(g_1, h_1Z; zg_2, h_2 w_z\inv Z; zz'g_3, h_3 w_{z'}\inv w_z\inv Z) \in X^{[3]}_{ijk}/(Z^3).
\]
For the same reasons as in  \ref{ss:basicGmodZ}, on fibres, the groupoid multiplication $\nu'$ takes the form 
\[
Q'_{ij;z}\big|_{(y_1,y_2)}\otimes Q'_{j_z k_z;z'}\big|_{(y_2,y_3)} \longrightarrow Q'_{ik;zz'}\big|_{(y_1,y_3)}.
\]
Similarly, under the identification \eqref{eq:decompY2'}, we have the following correspondence of elements in $(Y')^{[2]}$:
\begin{align*}
(y_1,y_2) & \longleftrightarrow (g,h\gamma\inv Z; g,h\gamma'\inv Z;z) \in X^{[2]}_{ij}/Z^2 \times \{z\}, \\
(y_2,y_3) & \longleftrightarrow (z\inv g,h w_z \gamma_z'\inv Z; z\inv g,h w_z \gamma_z''\inv Z;z')\in X^{[2]}_{j_z k_z}/Z^2 \times \{z'\}, \text{ and}\\
(y_1,y_3) & \longleftrightarrow (g,h\gamma\inv Z; g,h\gamma''\inv c_{z,z'} Z;zz')\in X^{[2]}_{ik}/Z^2 \times \{zz'\},
\end{align*}
where we have used the same notation as in \ref{ss:basicGmodZ}.

Recall that points in $Q'$ are of the form $[g,h,\tilde\gamma, \tilde\gamma',u]$ subject to the relations,
\begin{align*}
(g,h,\tilde{\gamma},\tilde{\gamma}',u) &\sim (g,h\beta,\tilde{\gamma}\tilde{\beta},\tilde{\gamma}'\tilde{\beta},\chi_{ij}(\tilde{\beta})^{-1}u) \quad \text{and} \\
 &\sim (g,h,\tilde{\gamma}\zeta,\tilde{\gamma}'\zeta',\chi_i(\zeta)\chi_j(\zeta')^{-1}u) \quad \text{and} \\
 &\sim (g,h,\tilde{\gamma}\, \widetilde{\exp}_i{\zeta_1},\tilde{\gamma}'\, \widetilde{\exp}_j{\zeta_2},e^{2\pi \cxi \, \ell(\langle \mu_i,\zeta_1\rangle-\langle\mu_j,\zeta_2\rangle)}u),
\end{align*}
for all $\beta \in G_{ij}$, $\zeta \in \mathcal{Z}_i$, $\zeta'\in \mathcal{Z}_j$ and $\zeta_1,\zeta_2$ in $\intlat_Z$.

Choose representatives in $P^\ell \times Z\to Y^{[2]}$ of the form
\begin{align*}
a&=[g,h,\tilde{\gamma},\tilde{\gamma}',u;z], \quad
b=[z^{-1}g,h\, w_z, \tilde{\gamma}_z',\tilde{\gamma}_z'',u';z'], \quad \text{and} \\
c&=[g,h,\tilde{\gamma},\tilde{c}_{z,z'}^{-1}\tilde{\gamma}'',uu';zz'],
\end{align*}
so that according to \eqref{eq:multGR}, 
\[
\nu(a\otimes b)=u_{z,z'}^{ijk} \, c.
\]
Below we show that $\nu$ is compatible with the natural $Z\times Z\times Z$-action on $Y^{[3]}$, and hence it descends to the desired groupoid multiplication $\nu'$.  We will need the following Lemma.
\begin{lemma} \label{l:multcompat}
Let $z\in Z$, $\zeta \in \intlat_Z$. Then $\, \widetilde{\exp}_{j_z}{(\Ad_{w_z^{-1}} \zeta-\zeta)}$ lies in $\mathcal{Z}_{j_z}$.
\end{lemma}
\begin{proof}
Since $\Ad_{w_z^{-1}} \, \widetilde{\exp}_j{\zeta} = \, \widetilde{\exp}_{j_z}{\Ad_{w_z^{-1}} \zeta}$, and $\Ad_{w_z^{-1}} \exp{ \zeta}=\exp{ \zeta}$ in $T\subset G_I$ (for any $I$), we have that $\, \widetilde{\exp}_{j_z}{ (\Ad_{w_z^{-1}} \zeta-\zeta)}$ lies in $\mathcal{Z}_{j_z}$. 
\end{proof}

\begin{proposition} \label{p:multcompat} 
Suppose $\ell$ is a multiple of $\ell_b$. 
Let $(z_1, z_2, z_3) \in Z^3$. For $a,b,c$ as above, let $a'=(z_1,z_2) \cdot a$, $b'=(z_2, z_3) \cdot b$ and $c'=(z_1, z_3)\cdot c$. Then $\nu(a'\otimes b')=u_{z,z'}^{ijk} \, c'$.
That is, the groupoid multiplication $\nu$ is compatible with the natural $Z^3$-action on $Y^{[3]}$. 
\end{proposition}
\begin{proof}
Let $\zeta_1, \zeta_2, \zeta_3 \in \intlat_Z$, and consider the natural action of $(\exp {\zeta_1},\exp {\zeta_2},\exp {\zeta_3})\in Z^3$:
\begin{align*}
a'=(\exp {\zeta_1},\exp {\zeta_2})\cdot a &= [g,h,\tilde{\gamma}\, \widetilde{\exp}_i{\zeta_1},\tilde{\gamma}'\, \widetilde{\exp}_j{\zeta_2},e^{2\pi \cxi \, \ell(\langle \mu_i,\zeta_1\rangle - \langle\mu_j,\zeta_2 \rangle)}u;z], \\
b'=(\exp {\zeta_2},\exp{ \zeta_3})\cdot b &= [z^{-1}g,h\, w_z, \tilde{\gamma}_z'\, \widetilde{\exp}_{j_z}{ \zeta_2},\tilde{\gamma}_z''\, \widetilde{\exp}_{k_z}{ \zeta_3},e^{2\pi \cxi \, \ell(\langle \mu_{j_z},\zeta_2\rangle - \langle\mu_{k_z},\zeta_3 \rangle)}u';z'], \\
c'=(\exp{ \zeta_1},\exp{ \zeta_3})\cdot c &= [g,h,\tilde{\gamma}\, \widetilde{\exp}_i^{2\pi\cxi \zeta_1},\tilde{c}_{z,z'}^{-1}\tilde{\gamma}''\, \widetilde{\exp}_k{ \zeta_3},e^{2\pi \cxi \, \ell(\langle \mu_i,\zeta_1\rangle - \langle\mu_k,\zeta_3 \rangle)}uu';zz']
\end{align*}

In order to apply the formula \eqref{eq:multGR} to $a'\otimes b'$, we will first modify the representative for $b'$.  Using Lemma \ref{l:multcompat} we note that
\begin{align*}
(&\exp{ \zeta_2}, \exp{ \zeta_3})\cdot b \\ 
&=[z^{-1}g,h\, w_z, \tilde{\gamma}_z'\, \widetilde{\exp}_{j_z}{ \zeta_2}\, \widetilde{\exp}_{j_z}{(\Ad_{w_z^{-1}} \zeta_2-\zeta_2)},\tilde{\gamma}_z''\, \widetilde{\exp}_{k_z}{ \zeta_3}\, \widetilde{\exp}_{k_z}{(\Ad_{w_z^{-1}} \zeta_3-\zeta_3)}, \\
	& \quad\chi^\ell_{j_z}\left(\, \widetilde{\exp}_{j_z}{(\Ad_{w_z^{-1}} \zeta_2-\zeta_2)}\right) \chi^\ell_{k_z}\left(\, \widetilde{\exp}_{k_z}{ (\Ad_{w_z^{-1}} \zeta_3-\zeta_3)}\right)^{-1}
e^{2\pi \cxi \, \ell(\langle \mu_{j_z},\zeta_2\rangle - \langle\mu_{k_z},\zeta_3 \rangle)}u';z'], \\
&=[z^{-1}g,h\, w_z, \left(\tilde{\gamma}'\, \widetilde{\exp}_{j_z}{ \zeta_2}\right)_z, \left(\tilde{\gamma}''\, \widetilde{\exp}_k{\zeta_3}\right)_z, \\
	&\quad e^{2\pi \cxi \, \ell(\langle \mu_{j_z},(\Ad_{w_z^{-1}} \zeta_2-\zeta_2)\rangle-\langle\mu_{k_z},(\Ad_{w_z^{-1}} \zeta_3-\zeta_3)\rangle)}
	e^{2\pi \cxi \, \ell(\langle \mu_{j_z},\zeta_2\rangle - \langle\mu_{k_z},\zeta_3 \rangle)}u';z'].
\end{align*}
Observe that the expression in the exponential in the second last factor simplifies to
\begin{align*}
\ell\langle \mu_{j_z}, &(\Ad_{w_z^{-1}} \zeta_2-\zeta_2)\rangle-\ell\langle\mu_{k_z},(\Ad_{w_z^{-1}} \zeta_3-\zeta_3)\rangle +\ell \langle \mu_{j_z},\zeta_2\rangle - \ell\langle\mu_{k_z},\zeta_3 \rangle \\
&= \ell\langle \mu_{j_z},\Ad_{w_z^{-1}} \zeta_2 \rangle - \ell\langle\mu_{k_z},(\Ad_{w_z^{-1}} \zeta_3 \rangle \\
&=\ell \langle \Ad_{w_z} \mu_{j_z}, \zeta_2 \rangle -\ell \langle \Ad_{w_z} \mu_{k_z}, \zeta_3 \rangle \\ 
&= \ell\langle (\mu_j-\mu_{j0}),\zeta_2 \rangle -  \ell\langle (\mu_k-\mu_{k0}),\zeta_3 \rangle
\end{align*}
where $\mu_{i0}$ denotes the vertex $z\cdot \mu_0=\lambda_i^\vee$ (and similarly for indices $j,k$),  which is a central element. Therefore,
\begin{align*}
(\exp{ \zeta_2},\exp{ \zeta_3})\cdot b = &
[z^{-1}g,h\, w_z, \left(\tilde{\gamma}'\, \widetilde{\exp}_{j_z}{ \zeta_2}\right)_z, \left(\tilde{\gamma}''\, \widetilde{\exp}_k{\zeta_3}\right)_z, \\
&	 e^{2\pi \cxi \, \ell(\langle (\mu_j-\mu_{j0}),\zeta_2 \rangle -  \langle (\mu_k-\mu_{k0}),\zeta_3 \rangle)}u';z'].
\end{align*}

Since $\ell$ is a multiple of the $\ell_b$, then $\ell\langle\mu_{l0},\zeta\rangle\in \ZZ$, for any $l$ and $\zeta \in \intlat_Z$; hence,
\[
b'=[z^{-1}g,h\, w_z, \left(\tilde{\gamma}'\, \widetilde{\exp}_{j_z}{ \zeta_2}\right)_z, \left(\tilde{\gamma}''\, \widetilde{\exp}_k{ \zeta_3}\right)_z, 
	 e^{2\pi \cxi \, \ell(\langle \mu_j,\zeta_2 \rangle -  \langle \mu_k,\zeta_3 \rangle)}u';z'].
\]
It follows that 
\[
\nu(a'\otimes b') = u_{z,z'}^{ijk} \, c',
\]
as required.
\end{proof}

Proposition \ref{p:multcompat} shows that the multiplication $\nu$ is compatible with the natural $Z^3$-action whenever $\ell$ is a multiple of $\ell_b$.  Hence it descends to the quotient $Q'$ and defines the groupoid multiplication $\nu'$.  Note that the associativity condition holds since it holds before passing to the quotient.

That the multiplication is compatible with the $G/Z$-action is easily verified. Indeed, let $fZ\in G/Z$ and consider
\begin{align*}
fZ\cdot a & = [fgf^{-1},f\, h,\tilde{\gamma},\tilde{\gamma}',u;z], \\
fZ \cdot  b & = [fz^{-1}gf^{-1},f\, h\, w_z, \tilde{\gamma}_z',\tilde{\gamma}_z'',u';z'], \\
&=[z^{-1}fgf^{-1},f\, h\, w_z, \tilde{\gamma}_z',\tilde{\gamma}_z'',u';z'], \\
fZ \cdot  c & = [fgf^{-1},f\, h,\tilde{\gamma},\tilde{c}_{z,z'}^{-1}\tilde{\gamma}'',uu';zz'].
\end{align*}
Then  $\mu'(fZ\cdot a \otimes fZ \cdot b)=u_{z,z'}^{ijk}fZ \cdot c$, as required.

This construction verifies the   `if' part of Theorem \ref{t:obs}, which we record in the following Proposition.
\begin{proposition} \label{p:equivbasicGmodZ}
Let $\ell=\ell_b$ in the construction above.  Then $(Y',Q',\nu')$ is a basic $G/Z$-equivariant  gerbe over $G/Z$.
\end{proposition}

\begin{remark} \label{r:gerbeonPO2} Note that the bundle gerbe from Remark \ref{r:gerbeonPO} representing the torsion class in $H^3(\PO(4n);\ZZ)$ is strongly equivariant with respect to the conjugation action. Therefore, the two non-isomorphic  basic bundle gerbes for $\PO(4n)$ each admit equivariant extensions (at even levels $\ell$).
\eoe
\end{remark}

\subsection{Twisting by a character of $Z$} \label{ss:twist}

The gerbe multiplication on the equivariant bundle gerbe  on $G/Z$ from Section \ref{ss:basicGmodZ} may be \emph{twisted} by a character of $Z$ in a sense we presently describe.  Recall that $H^3_{G/Z}(G/Z;\ZZ)$ fits in an exact sequence,
\begin{equation} \label{eq:exactseq}
0 \longrightarrow \Hom(Z,\U(1)) \longrightarrow H^3_{G/Z}(G/Z;\ZZ) \longrightarrow H^3(G/Z;\ZZ).
\end{equation}
 This can be seen, for instance, from the Serre spectral sequence for the Borel construction 
\[
G/Z \to G/Z \times_{G/Z} E(G/Z) \to B(G/Z)
\]
 and the fact that all central $\U(1)$-extensions over $G/Z$ can be made $G/Z$-equivariant. We view $\Hom(Z,\U(1)) \cong H^2(G/Z;\ZZ) \cong H^3(B(G/Z);\ZZ)$. Classes in $H^3(B(G/Z);\ZZ)$ are represented by $G/Z$-equivariant bundles over a point (e.g.\ see \cite[Example 2.5]{meinrenken2003basic}); therefore, they pullback to (non-equivariantly) trivial bundle gerbes $\calH_\chi$ on $G/Z$.  In particular, given a basic equivariant  bundle gerbe $\calG'$ over $G/Z$, the tensor product $\calG' \otimes \calH_\chi$ is also a basic equivariant  bundle gerbe.  Below we sketch a construction of the equivariant bundle gerbes $\calH_\chi$.   
 
 Let $\chi:Z\to \U(1)$ be a character and let $\widehat{G}=G\times_Z \CC$ denote the corresponding line bundle over $G/Z$.  Consider the  submersion $\pr_2:G/Z \times G/Z \to G/Z$ given by projection onto the second factor, with $G/Z$-action given by
 \[
 g\cdot(k,h)=(kg\inv ,\Ad_{g}h).
 \]
  Let $\varphi:(G/Z)^2 \to G/Z$ be given by 
 $\varphi(k_0,k_1)=k_0k_1\inv$ and consider the line bundle $L_\chi=\varphi^*\widehat{G} \times G/Z \to (G/Z \times G/Z)^{[2]}=(G/Z)^3$,  with projection
\[
([\gamma,\lambda],k_0,k_1,h) \mapsto (k_0,k_1,h).
\] 
The $G/Z$-action naturally lifts to $L_\chi$,
\[
g\cdot ([\gamma,\lambda],k_0,k_1,h) = ([\gamma,\lambda],k_0 g\inv , k_1 g\inv, \Ad_{g} h).
\] 
The gerbe multiplication $\tau:L_\chi\big|_{(k_0,k_1,h)} \otimes L_\chi\big|_{(k_1,k_2,h)} \to L_\chi\big|_{(k_0,k_2,h)}$ is induced from the natural monoid operation on $\widehat{G}$,  given by 
\[
([\gamma,\lambda], k_0, k_1, h) \otimes ([\gamma', \lambda'], k_1,k_2,h) \mapsto ([\gamma\gamma', \lambda\lambda'], k_0, k_1, h),
\]
which is easily seen to be compatible with the $G/Z$-action.  This shows most of the following Proposition.

\begin{proposition} \label{p:twists}
The bundle gerbes $\calH_\chi=(G/Z \times G/Z, L_\chi, \tau)$ constructed above are  strongly equivariant bundle gerbes whose DD-classes are the elements in the kernel of the natural map $H^3_{G/Z}(G/Z;\ZZ) \to H^3(G/Z;\ZZ)$.
\end{proposition}
\begin{proof}
It remains to verify the statement about DD-classes.  Notice that the quotient bundle gerbes $\calH_\chi/(G/Z)$ may be viewed as lifting bundle gerbes for the principal $G/Z$-bundle 
\[
G/Z \times E(G/Z) \to G/Z \times_{G/Z} E(G/Z)
\]
 corresponding to the central $\U(1)$-extensions $\U(1) \to G\times_Z \U(1) \to G/Z$ determined by $\chi:Z\to \U(1)$.  
The construction of $\calH_\chi$ above is simply the pullback of this lifting bundle gerbe to $G/Z$. That is, its DD-class lies in the image of the map $H^3(B(G/Z);\ZZ) \to H^3_{G/Z}(G/Z;\ZZ)$, as claimed.
\end{proof}

%

\end{document}